\documentclass[12pt,reqno]{amsart}
\usepackage{amsmath, amsfonts, amssymb, amsthm, hyperref, cite}
\usepackage{bm}
\allowdisplaybreaks[4]
\def\l{\left}
\def\r{\right}
\def\bg{\bigg}
\def\({\bg(}
\def\){\bg)}
\def\t{\text}
\def\f{\frac}

\def\eq{\equiv}

\def\Z{\mathbb Z}

\def\<{\langle}
\def\>{\rangle}
\def\1{{\bf 1}}

\theoremstyle{plain}
\newtheorem{theorem}{Theorem}[section]
\newtheorem{conjecture}{Conjecture}

\newtheorem{lemma}{Lemma}[section]

\theoremstyle{definition}

\newtheorem*{Acks}{Acknowledgments}
\theoremstyle{remark}

\numberwithin{equation}{section}

\begin{document}
\title[Two supercongruences involving truncated hypergeometric series]{Two supercongruences involving truncated hypergeometric series}
\author[Wei Xia]{Wei Xia}
\address[Wei Xia]{Department of Mathematics, Nanjing
University, Nanjing 210093, People's Republic of China}
\email{wxia@smail.nju.edu.cn}
\author[Chen Wang]{Chen Wang*}
\address[Chen Wang]{Department of Applied Mathematics, Nanjing Forestry
University, Nanjing 210037, People's Republic of China}
\email{cwang@smail.nju.edu.cn}
\begin{abstract}
In this paper, we mainly establish two supercongruences involving truncated hypergeometric series by using some hypergeometric transformation formulas. The first supercongruence confirms a recent conjecture of the second author. The second supercongruence confirms a conjecture of Guo, Liu and Schlosser partially, and gives a parametric extension of a supercongruence of Long and Ramakrishna.
\end{abstract}
\thanks{*Corresponding author}
\maketitle

\section{Introduction}

Let $n\in\mathbb{N}=\{0,1,2,\ldots\}$ and $(a)_n=a(a+1)\cdots(a+n-1)$ denote the Pochhammer symbol. For $r\in\mathbb{N}$ and $a_0,\ldots,a_r, b_1,\ldots,b_r,z\in\mathbb{C}$, the hypergeometric series ${}_{r+1}F_{r}$ are defined by
$$
{}_{r+1}F_r\bigg[\begin{matrix}a_0,&a_1,&\ldots,&a_r\\ &b_1,&\ldots,&b_r\end{matrix}\bigg|\ z\bigg]=\sum_{k=0}^{\infty}\f{(a_0)_k\cdots(a_r)_k}{(b_1)_k\cdots(b_r)_k}\cdot\f{z^k}{k!}.
$$
Partial sums of the hypergeometric series are usually called truncated hypergeometric series which are given by
$$
{}_{r+1}F_r\bigg[\begin{matrix}a_0,&a_1,&\ldots,&a_r\\ &b_1,&\ldots,&b_r\end{matrix}\bigg|\ z\bigg]_n=\sum_{k=0}^{n}\f{(a_0)_k\cdots(a_r)_k}{(b_1)_k\cdots(b_r)_k}\cdot\f{z^k}{k!}.
$$
When any of the $\alpha_i$ is a negative integer and none of the $\beta_i$ are negative integers larger than all $\alpha_j$, the above hypergeometric series terminates, and are also called the truncated hypergeometric series. In the past decades, supercongruences involving truncated hypergeometric series have been widely studied by different authors (cf. \cite{GuoLiu,GLS,GS1,GS,Liu2017,Liu2021,LR,MaoPan2022,PTW,Sun2011,VanHamme,Wang,WP,WS}).

Guo, Liu and Schlosser \cite[Theorem 1]{GLS} established the following supercongruence concerning truncated hypergeometric series ${}_6F_5$:
\begin{equation}\label{GLS6F5}
\sum_{k=0}^{p-1}(10k+r)\f{(\f r 5)_k^5}{k!^5}\eq0\pmod{p^4},
\end{equation}  
where $r\leq 1$ is an odd integer coprime with $5$ and $p\geq (5-r)/2$ is a prime such that $p\eq 2r\pmod 5$. Recently, under the same conditions of \eqref{GLS6F5}, the second author \cite{Wang} extended \eqref{GLS6F5} to the modulus $p^5$ case as follows:
\begin{equation}\label{GLSth1exeq}
\sum_{k=0}^{p-1}(10k+r)\f{(\f{r}{5})_k^5}{(1)_k^5}\eq\f{12p^4}{25}\cdot\f{\Gamma_p(\f r5)^4}{\Gamma_p(\f{2r}5)^2\Gamma_p(\f12+\f{3r}{10})\Gamma_p(\f12-\f r{10})^3}\sum_{k=0}^{(1-r)/2}\f{(\f{r-1}{2})_k(\f r5)_k^3}{(1)_k(\f{2r}5)_k^2(\f12+\f{3r}{10})_k}\pmod{p^5}.
\end{equation}
The first purpose of this paper is to establish a variant of \eqref{GLSth1exeq} which confirms a recent conjecture of the second author \cite[Conjecture 4.1]{Wang}.
\begin{theorem}\label{theorem1}
Let $r\leq 1$ be an odd integer coprime with $5$. Let $p$ be an odd prime such that $p\eq r\pmod{5}$ and $p\geq (5-3r)/2$. Then
\begin{equation}\label{conj1eq}
\sum_{k=0}^{p-1}(10k+r)\f{(\f{r}{5})_k^5}{(1)_k^5}\eq\f{p\Gamma_p(\f r5)^4}{\Gamma_p(\f{2r}5)^2\Gamma_p(\f12+\f{3r}{10})\Gamma_p(\f12-\f r{10})^3}\sum_{k=0}^{(1-r)/2}\f{(\f{r-1}{2})_k(\f r5)_k^3}{(1)_k(\f{2r}5)_k^2(\f12+\f{3r}{10})_k}\pmod{p^5}.
\end{equation}
\end{theorem}
The second author proved \eqref{GLSth1exeq} by using a ${}_7F_6$ transformation formula due to Liu \cite[Lemma 2.6]{Liu2021}. However, we cannot prove \eqref{conj1eq} in the same way. Our proof relies on a new ${}_7F_6$ transformation formula deduced from both Liu's transformation formula and Whipple's ${}_4F_3$ transformation formula \cite[Theorem 3.3.3]{Andrews}. Based on many computations, we point out that the congruence \eqref{conj1eq} does not hold modulo $p^6$ in general.

In 1997, Van Hamme \cite{VanHamme} conjectured 13 supercongruences concerning truncated hypergeometric series motivated by Ramanujan-type series. For example, he asserts that for any prime $p\geq5$,
\begin{equation}\label{VanHammeconj}
\sum_{k=0}^{p-1}(6k+1)\f{(\f13)_k^6}{(1)_k^6}\eq\begin{cases}-p\Gamma_p\l(\f13\r)^9\pmod{p^4}\quad&\t{if}\ p\eq1\pmod{6},\\ 0\pmod{p^4}\quad&\t{if}\ p\eq5\pmod{6},\end{cases}
\end{equation}
where $\Gamma_p(x)$ is the $p$-adic Gamma function introduced by Morita \cite{Morita}. In 2016, the supercongruence \eqref{VanHammeconj} was confirmed by Long and Ramakrishna \cite[Theorem 2]{LR} in the following strengthening form:
\begin{equation}\label{LR}
\sum_{k=0}^{p-1}(6k+1)\f{(\f13)_k^6}{(1)_k^6}\eq\begin{cases}-p\Gamma_p\l(\f13\r)^9\pmod{p^6}\quad&\t{if}\ p\eq1\pmod{6},\\ -\f{10p^4}{27}\Gamma_p\l(\f13\r)^9\pmod{p^6}\quad&\t{if}\ p\eq5\pmod{6}.\end{cases}
\end{equation}
Later, Liu\cite[Theorem1.2]{Liu2021} derived a similar result as follows:
\begin{equation}\label{Liucon}
\sum_{k=0}^{p-1}(6k-1)\f{(-\f13)_k^6}{(1)_k^6}\eq\begin{cases}140p^4\Gamma_p\l(\f23\r)^9\pmod{p^5}\quad&\t{if}\ p\eq1\pmod{6},\\ 378p\Gamma_p\l(\f23\r)^9\pmod{p^5}\quad&\t{if}\ p\eq5\pmod{6},\end{cases}
\end{equation}
where $p\geq5$ is a prime. Recently, Guo, Liu and Schlosser \cite{GLS} gave a parametric extension of \eqref{LR} and \eqref{Liucon}. In fact, they presented the following surprisingly beautiful conjecture.
\begin{conjecture}\label{conjecture1}Let $r\leq 1$ be an integer coprime with $3$. \rm{(i)} If $p\geq5$ is a prime such that $p\eq-r\pmod{3}$ and $p\geq3-r$, then
\begin{align}\label{GLSconj1eq}
\sum_{k=0}^{p-1}(6k+r)\f{(\f{r}3)_k^6}{(1)_k^6}\eq&\ \f{(-1)^{r}10p^4}{27}\cdot\f{\Gamma_p(\frac{r}{3})^6}{\Gamma_p(\f{2r}3)^3}\sum_{k=0}^{1-r}\f{(r-1)_k(\f{r}3)_k^3}{(1)_k(\f{2r}{3})_k^3}\pmod{p^6}.
\end{align}
\rm{(ii)} If $p\geq7$ is a prime such that $p\eq r\pmod{3}$ and $p\geq3-2r$, then
\begin{align}\label{GLSconj2eq}
\sum_{k=0}^{p-1}(6k+r)\f{(\f{r}3)_k^6}{(1)_k^6}\eq&\ (-1)^{r+1}p\cdot\f{\Gamma_p(\frac{r}{3})^6}{\Gamma_p(\f{2r}3)^3}\sum_{k=0}^{1-r}\f{(r-1)_k(\f{r}3)_k^3}{(1)_k(\f{2r}{3})_k^3}\pmod{p^6}.
\end{align}
\end{conjecture}
The congruence \eqref{GLSconj1eq}, restricted to modulus $p^5$, was confirmed by Guo, Liu and Schlosser \cite{GLS}, and completely solved by the second author \cite{Wang}. The second purpose of this article is to confirm the modulus $p^5$ case of the congruence \eqref{GLSconj2eq}.
\begin{theorem}\label{theorem2}
The supercongruence \eqref{GLSconj2eq}, restricted to modulus $p^5$ case, is true.
\end{theorem}
The second author proved \eqref{GLSconj1eq} by using the fifth primitive root of unity $\zeta$. However, when we attempted to prove \eqref{GLSconj2eq} modulo $p^5$ in the same way as in the proof of \eqref{GLSconj1eq}, the hidden symmetries of some harmonic sums are missing. To avoid this obstruction, we use the fourth primitive root of unity $i$ instead of $\zeta$. It appears to be rather difficult to prove \eqref{GLSconj2eq} completely in this way. We hope that an interested reader will make some progress on it.

In our proofs, the key ingredients are some hidden symmetric relations and  hypergeometric transformation formulas. The rest of this paper is organized as follows. Section 2 is devoted to recalling some necessary transformation formulas of hypergeometric series and some properties of the $p$-adic Gamma functions. We will prove Theorems \ref{theorem1} and \ref{theorem2} in Sections 3 and 4, respectively.

\section{Preliminary results}

\begin{lemma}[Whipple {\cite[Theorem 3.3.3]{Andrews}}]\label{lem4f3}
Let $n$ be nonnegative integers. Then
\begin{equation}\label{4f3}
{}_{4}F_3\bigg[\begin{matrix}-n,&a,&b,&c\\ &d,&e,&f\end{matrix}\bigg|\ 1\bigg]
=\frac{(e-a)_n(f-a)_n}{(e)_n(f)_n}\cdot{}_{4}F_3\bigg[\begin{matrix}-n,&a,&d-b,&d-c\\ &d,&a+1-n-e,&a+1-n-f\end{matrix}\bigg|\ 1\bigg],
\end{equation}
where $a+b+c-n+1=d+e+f$.
\end{lemma}
Combining the above transformation formula and Whipple's ${}_7F_6$ transformation formula \cite[Theorem 3.4.5]{Andrews}, Liu \cite[Lemma 2.6]{Liu2021} obtained the following result: For nonnegative integers $n$ and $m$, we have
\begin{align}\label{liueq}
&{}_7F_6\bigg[\begin{matrix}t,&1+\f12t,&-n,&t-a,&t-b,&t-c,&1-t-m+n+a+b+c\\ &\f12t,&1+t+n,&1+a,&1+b,&1+c,&2t+m-n-a-b-c\end{matrix}\bigg|\ 1\bigg]\notag\\
=\ &{}_4F_3\bigg[\begin{matrix}-m,&-n,&a+b+c+1-m-2t,&a+b+c+1+n-m-t\\ &a+b+1-m-t,&a+c+1-m-t,&b+c+1-m-t\end{matrix}\bigg|\ 1\bigg]\notag\\
&\times\f{(1+t)_n(a+b+1-m-t)_n(a+c+1-m-t)_n(b+c+1-m-t)_n}{(1+a)_n(1+b)_n(1+c)_n(a+b+c+1-m-2t)_n}.
\end{align}
With the help of  \eqref{4f3} and \eqref{liueq}, we derive the following transformation formula.
\begin{lemma}
For nonnegative integers $n$ and $m$, we have
\begin{align}\label{new}
&{}_7F_6\bigg[\begin{matrix}t,&1+\f12t,&-n,&t-a,&t-b,&t-c,&1-t-m+n+a+b+c\\ &\f12t,&1+t+n,&1+a,&1+b,&1+c,&2t+m-n-a-b-c\end{matrix}\bigg|\ 1\bigg]\notag\\
=\ &{}_4F_3\bigg[\begin{matrix}-m,&-n,&t-b,&-n-b\\ &a+c+1-m-t,&t-n-a-b,&t-n-b-c\end{matrix}\bigg|\ 1\bigg]\notag\\
&\times\f{(1+t)_n(a+b+1-t)_n(a+c+1-m-t)_n(b+c+1-t)_n}{(1+a)_n(1+b)_n(1+c)_n(a+b+c+1-m-2t)_n}.
\end{align}
\end{lemma}
\begin{proof}Setting $a\rightarrow-m$, $n\rightarrow n$, $b\rightarrow a+b+c+1-m-2t$, $c\rightarrow a+b+c+1+n-m-t$, $d\rightarrow a+c+1-m-t$, $e\rightarrow a+b+1-m-t$ and $f\rightarrow b+c+1-m-t$ in \eqref{4f3}, we have
\begin{align*}
&{}_4F_3\bigg[\begin{matrix}-m,&-n,&a+b+c+1-m-2t,&a+b+c+1+n-m-t\\ &a+b+1-m-t,&a+c+1-m-t,&b+c+1-m-t\end{matrix}\bigg|\ 1\bigg]\\
=\ &{}_4F_3\bigg[\begin{matrix}-m,&-n,&t-b,&-n-b\\ &a+c+1-m-t,&t-n-a-b,&t-n-b-c\end{matrix}\bigg|\ 1\bigg]\\
&\times\frac{(a+b+1-t)_n(b+c+1-t)_n}{(a+b+1-m-t)_n(b+c+1-m-t)_n}.
\end{align*}
Substituting the above equation into the right-hand side of \eqref{liueq}, we immediately get the desired result.
\end{proof}
Let $p$ be an odd prime and $\mathbb{Z}_p$ be the ring of $p$-adic integers. Let $\nu_p(\cdot)$ denote the $p$-adic order and $|x|_p=p^{-\nu_p(x)}$ the $p$-adic norm. We now recall the definition of Morita's $p$-adic Gamma function $\Gamma_p$. For each integer $n\geq1$, define
$$\Gamma_p(n):=(-1)^n\prod\limits_{1\leq k<n\atop p\nmid k}k.$$
Moreover, set $\Gamma_p(0):=1$. In general, for any $x\in\mathbb{Z}_p$, the $p$-adic Gamma function $\Gamma_p(x)$ is defined as follows:
$$\Gamma_p(x):=\lim\limits_{|n-x|_p\rightarrow0}\Gamma_p(n).$$
Throughout this paper, we use $\<x\>_{p^n}$ to denote the least nonnegative residue of $x$ modulo $p^n$. In order to prove Theorems \ref{theorem1} and \ref{theorem2}, we need some properties of the $p$-adic Gamma function.
\begin{lemma}[Robert {\cite[p. 369]{Robert00}}]\label{robert}
Let $p$ be an odd prime. Then, for $x,y\in\mathbb{Z}_p$, we have
\begin{align*}
&\Gamma_p(x)\Gamma_p(1-x)=(-1)^{\<-x\>_p-1},\\
&\Gamma_p(x)\eq\Gamma_p(y)\pmod{p^n}\ \ \t{for}\ \ x\eq y\pmod{p^n},\\
&\f{\Gamma_p(x+1)}{\Gamma_p(x)}=\begin{cases}-x\ \ \t{if}&\ \ \nu_p(x)=0,\\
-1\ \ \t{if}&\ \ \nu_p(x)>0.
\end{cases}
\end{align*}
\end{lemma}
Set $G_k(a):=\Gamma_p^{(k)}(a)/\Gamma_p(a)$. In particular, $G_0(a)=1$. The following lemma is due to Long and Ramakrishna \cite[Theorem 14]{LR}, which provides the $p$-adic expansions of $\Gamma_p(x)$.
\begin{lemma}\label{Long}
Let $p\geq7$ be a prime and $a,m\in\mathbb{Z}_p$. Then
$$\Gamma_p(a+mp)\equiv\Gamma_p(a)\sum_{k=0}^{3}\frac{G_k(a)}{k!}(mp)^k\pmod{p^4}.$$
\end{lemma}

\begin{lemma}\label{g}
Let $p\geq7$ be a prime and $a,m\in\mathbb{Z}_p$. Then
$$\Gamma_p(a+mp)\Gamma_p(a-mp)\equiv\Gamma_p(a)^2\bigg(1+m^2p^2\sum\limits_{1\leq j\leq\<-a\>_{p^2}\atop p\nmid j}\frac{1}{j^2}\bigg)\pmod{p^4}.$$
\end{lemma}
\begin{proof}
In view of Lemma \ref{Long}, we have
$$\Gamma_p(a+mp)\equiv\Gamma_p(a)\sum_{k=0}^{3}\frac{G_k(a)}{k!}(mp)^k\pmod{p^4}$$
and
$$\Gamma_p(a-mp)\equiv\Gamma_p(a)\sum_{k=0}^{3}\frac{G_k(a)}{k!}(-mp)^k\pmod{p^4}.$$
Then
$$\Gamma_p(a+mp)\Gamma_p(a-mp)\equiv\Gamma_p(a)^2(1+m^2p^2(G_2(a)-G_1(a)^2))\pmod{p^4}.$$
Therefore it remains to prove that
$$G_2(a)-G_1(a)^2\equiv\sum\limits_{1\leq j\leq\<-a\>_{p^2}\atop p\nmid j}\frac{1}{j^2}\pmod{p^2}.$$
Recall that Wang and Pan \cite[Lemma 2.4]{WP} deduced the explicit formula of $G_1(a)$ modulo $p^2$,
$$G_1(a)\equiv G_1(0)+\sum\limits_{1\leq j<\<a\>_{p^2}\atop p\nmid j}\frac{1}{j}\pmod{p^2}.$$
Furthermore, according to Pan, Tauraso and Wang \cite[Theorem 4.1]{PTW}, we have
$$G_2(a)\equiv G_2(0)+2G_1(0)\sum\limits_{1\leq j<\<a\>_{p^2}\atop p\nmid j}\frac{1}{j}+2\sum\limits_{1\leq i<j<\<a\>_{p^2}}\frac{1}{ij}\pmod{p^2}.$$
Combining the above two congruences and noting that $G_2(0)=G_1(0)^2$ (cf. \cite[Lemma 2.3]{WP}), we arrive at
$$G_2(a)-G_1(a)^2\equiv-\sum\limits_{1\leq j<\<a\>_{p^2}\atop p\nmid j}\frac{1}{j^2}\pmod{p^2}.$$
It is known (cf. \cite[Corollary 1(a)]{Slavutsky}) that
$$\sum\limits_{1\leq j\leq p^2\atop p\nmid j}\frac{1}{j^2}\equiv0\pmod{p^2}.$$
Then
$$-\sum\limits_{1\leq j<\<a\>_{p^2}\atop p\nmid j}\frac{1}{j^2}\equiv\sum\limits_{1\leq j\leq p^2-\<a\>_{p^2}}\frac{1}{(p^2-j)^2}\equiv\sum\limits_{1\leq j\leq p^2-\<a\>_{p^2}}\frac{1}{j^2}=\sum\limits_{1\leq j\leq \<-a\>_{p^2}}\frac{1}{j^2}\pmod{p^2}.$$
This proves the desired result.
\end{proof}
\section{Proof of Theorem \ref{theorem1}}

Throughout this paper, we use $i$ to denote the imaginary unit.
\begin{lemma}[Wang {\cite[Lemma 3.1]{Wang}}]\label{lemw}
Let $p$ be an odd prime. Then, for $u,v\in\Z_p$ and $k\in\{0,1,\ldots,\<-u\>_p\}$,
\begin{equation*}
(u+vp)_k(u-vp)_k(u+vip)_k(u-vip)_k\eq(u)_k^4\l(1+v^4p^4\sum_{j=0}^{k-1}\f{1}{(u+j)^4}\r)\pmod{p^5}.
\end{equation*}
\end{lemma}
\begin{lemma}\label{lemma31}
Under the assumptions of Theorem \ref{theorem1}, we have
$$\sum_{k=0}^{(p-r)/5}(10k+r)\frac{(\frac{r}{5})_k^5}{(1)_k^5}\equiv0\pmod p.$$
\end{lemma}
\begin{proof}Setting $m=\frac{1-r}{2}$, $t=\f{r}{5}$, $n=\frac{p-r}{5}$, $a=\f{r-5}{10}$, $b=-\f{p}{5}$ and $c=-\f{ip}{5}$ in \eqref{new}, we have
\begin{align}\label{n1}
&{}_7F_6\bigg[\begin{matrix}\frac{r}{5},&1+\frac{r}{10},&\frac{r-p}{5},&\frac{r+5}{10},&\frac{r+p}{5},&\frac{r+ip}{5},&\frac{r-ip}{3}\\ &\f r{10},&1+\f p5,&\f{r+5}{10},&1-\f p5,&1-\f{ip}{5},&1+\f{ip}{5}\end{matrix}\bigg|\ 1\bigg]\notag\\
=\ &{}_4F_3\bigg[\begin{matrix}\f{r-1}{2},&\f{r-p}{5},&\f{r+p}{5},&\f{r}{5}\\ &\f{2r-ip}{5},&\f{1}{2}+\f{3r}{10},&\f{2r+ip}{5}\end{matrix}\bigg|\ 1\bigg]\notag\\
&\times\f{(1+\f{r}{5})_n(\f12-\f{r}{10}-\f{p}5)_n(\f{2r}{5}-\f{ip}5)_n(1-\f{r}5-\f{(1+i)p}{5})_n}{(\f12+\f{r}{10})_n(1-\frac{p}{5})_n(1-\f{ip}{5})_n(\f{r}{5}-\f{(1+i)p}{5})_n}.
\end{align}
Since $r\leq1$ and $p\geq (5-3r)/2$, we have $\f{1-r}{2}\leq\f{p-r}{5}\leq\f{p-1}3$. Thus
\begin{align*}
&{}_7F_6\bigg[\begin{matrix}\frac{r}{5},&1+\frac{r}{10},&\frac{r-p}{5},&\frac{r+5}{10},&\frac{r+p}{5},&\frac{r+ip}{5},&\frac{r-ip}{3}\\ &\f r{10},&1+\f p5,&\f{r+5}{10},&1-\f p5,&1-\f{ip}{5},&1+\f{ip}{5}\end{matrix}\bigg|\ 1\bigg]\notag\\
=\ &{}_6F_5\bigg[\begin{matrix}\frac{r}{5},&1+\frac{r}{10},&\frac{r-p}{5},&\frac{r+p}{5},&\frac{r+ip}{5},&\frac{r-ip}{3}\\ &\f r{10},&1+\f p5,&1-\f p5,&1-\f{ip}{5},&1+\f{ip}{5}\end{matrix}\bigg|\ 1\bigg]_{\f{p-r}{5}}\notag\\
\equiv\ &\f1r\sum_{k=0}^{(p-r)/5}(10k+r)\f{(\f r5)_k^5}{(1)_k^5}\pmod{p}.
\end{align*}
Noting that
$$\l(1+\f{r}{5}\r)_n=\l(1+\f{r}{5}\r)\l(2+\f{r}{5}\r)\cdots\f{p}{5}\equiv0\pmod{p}$$
and
\begin{align*}
&\l(\f12+\f{r}{10}\r)_n\l(1-\frac{p}{5}\r)_n\l(1-\f{ip}{5}\r)_n\l(\f{r}{5}-\f{(1+i)p}{5}\r)_n\\
&\times\l(\f{2r-ip}{5}\r)_n\l(\f{1}{2}+\f{3r}{10}\r)_n\l(\f{2r+ip}{5}\r)_n\not\equiv0\pmod{p},
\end{align*}
we get that the right-hand side of \eqref{n1} is equivalent to 0 modulo $p$. Therefore
$$\sum_{k=0}^{(p-r)/5}(10k+r)\frac{(\frac{r}{5})_k^5}{(1)_k^5}\equiv0\pmod p.$$
\end{proof}

\begin{lemma}\label{lemma32}
Under the assumptions of Theorem \ref{theorem1}, we have
$$\sum_{k=0}^{(p-r)/5}(10k+r)\frac{(\frac{r}{5})_k^5}{(1)_k^5}\left(\sum_{j=0}^{k-1}\frac{1}{(r/5+j)^4}-\sum_{j=0}^{k-1}\frac{1}{(1+j)^4}\right)\equiv0\pmod p.$$
\end{lemma}
\begin{proof}Since both $p$ and $r$ are odd, we have $(p-r)/5$ is even and hence
 $$\left(\frac{r}{5}\right)_{(p-r)/5}=(-1)^{(p-r)/5}\left(1-\frac{p}{5}\right)_{(p-r)/5}\equiv(1)_{(p-r)/5}\pmod{p}.$$
It is routine to check that
\begin{align*}
&\sum_{k=0}^{(p-r)/5}(10k+r)\frac{(\frac{r}{5})_k^5}{(1)_k^5}\sum_{j=0}^{k-1}\frac{1}{(1+j)^4}\\
=\ &\sum_{k=0}^{(p-r)/5}\left(10\left(\frac{p-r}{5}-k\right)+r\right)\frac{(\frac{r}{5})_{(p-r)/5-k}^5}{(1)_{(p-r)/5-k}^5}\sum_{j=0}^{(p-r)/5-k-1}\frac{1}{(1+j)^4}\\
=\ &\frac{(\frac{r}{5})_{(p-r)/5}^5}{(1)_{(p-r)/5}^5}\sum_{k=0}^{(p-r)/5}(2p-10k-r)\frac{(\frac{r}{5}-\frac{p}{5})_k^5}{(1-\frac{p}{5})_k^5}\sum_{j=k}^{(p-r)/5-1}\frac{1}{((p-r)/5-j)^4}\\
\equiv\ &-\sum_{k=0}^{(p-r)/5}(10k+r)\frac{(\frac{r}{5})_k^5}{(1)_k^5}\sum_{j=k}^{(p-r)/5-1}\frac{1}{(r/5+j)^4}\pmod{p}.
\end{align*}
In view of Lemma \ref{lemma31}, we arrive at
\begin{align*}
&\sum_{k=0}^{(p-r)/5}(10k+r)\frac{(\frac{r}{5})_k^5}{(1)_k^5}\left(\sum_{j=0}^{k-1}\frac{1}{(r/5+j)^4}-\sum_{j=0}^{k-1}\frac{1}{(1+j)^4}\right)\\
\equiv\ &\sum_{k=0}^{(p-r)/5}(10k+r)\frac{(\frac{r}{5})_k^5}{(1)_k^5}\sum_{j=0}^{(p-r)/5-1}\frac{1}{(r/5+j)^4}\\
\equiv\ &0\pmod{p}.
\end{align*}
\end{proof}
\begin{lemma}\label{lemma33}
Under the assumptions of Theorem \ref{theorem1}, modulo $p^2$, we have
$$\sum_{j=0}^{(p-r)/5-1}\f{1}{(\f{2r}5+j)^2}-\sum_{j=0}^{(p-r)/5-1}\f{1}{(1+j)^2}+2\sum\limits_{1\leq j\leq\<-\f{r}{5}\>_{p^2}\atop p\nmid j}\f{1}{j^2}-\sum\limits_{1\leq j\leq\<\f{r}{10}-\f12\>_{p^2}\atop p\nmid j}\f{1}{j^2}\equiv\sum_{j=0}^{(-r-1)/2}\f{1}{(\f12+\f{r}{10}+j)^2}.$$
\end{lemma}
\begin{proof}Let $\alpha=\<-\frac{r}{5}\>_{p^2}$, $a_0=\<-\frac{r}{5}\>_{p}$ and $a_1=(\alpha-a_0)/p$. Since $p\equiv r\pmod 5$, we have $a_0=\f{p-r}5$.
It is easy to verify that
$$\sum_{j=0}^{(p-r)/5-1}\frac{1}{(\f{2r}{5}+j)^2}\equiv\sum_{j=0}^{(p-r)/5-1}\frac{1}{(-2\alpha+j)^2}=\sum_{j=2\alpha+1-a_0}^{2\alpha}\f1{j^2}\pmod{p^2}$$
and
\begin{align*}
\sum\limits_{1\leq j\leq\<-r/5\>_{p^2}\atop p\nmid j}\f1{j^2}-\sum_{j=1}^{(p-r)/5}\f1{j^2}=\ &\sum\limits_{a_0+1\leq j\leq a_0+a_1p\atop p\nmid j}\f1{j^2}=\sum\limits_{a_0+1+a_1p\leq j\leq a_0+2a_1p\atop p\nmid j}\f1{(j-a_1p)^2}\\
\equiv\ &\sum\limits_{a_0+1+a_1p\leq j\leq a_0+2a_1p\atop p\nmid j}\left(\f1{j^2}+\f{2a_1p}{j^3}\right)\\
\equiv\ &\sum\limits_{a_0+1+a_1p\leq j\leq a_0+2a_1p\atop p\nmid j}\f1{j^2}\\
=\ &\sum\limits_{\alpha+1\leq j\leq 2\alpha-a_0\atop p\nmid j}\f1{j^2}\pmod{p^2}.
\end{align*}
Therefore
$$\sum_{j=0}^{(p-r)/5-1}\frac{1}{(\f{2r}{5}+j)^2}-\sum_{j=0}^{(p-r)/5-1}\frac{1}{(1+j)^2}+2\sum\limits_{1\leq j\leq\<-r/5\>_{p^2}\atop p\nmid j}\frac{1}{j^2}\equiv\sum\limits_{1\leq j\leq2\alpha\atop p\nmid j}\frac{1}{j^2}\pmod{p^2}.$$
Without loss of generality, we may assume that $p\equiv1\pmod5$ since we can handle the other cases $p\equiv2,3,4\pmod5$ in a similar way. Under this circumstance, we have $\alpha=\<-\f{r}{5}\>_{p^2}=\f{p^2-r}5$ and $\<\f{r}{10}-\f12\>_{p^2}=\f{p^2-1-\alpha}{2}.$ Thus we obtain

\begin{align*}
&\sum_{j=0}^{(p-r)/5-1}\f{1}{(\f{2r}5+j)^2}-\sum_{j=0}^{(p-r)/5-1}\f{1}{(1+j)^2}+2\sum\limits_{1\leq j\leq\<-\f{r}{5}\>_{p^2}\atop p\nmid j}\f{1}{j^2}-\sum\limits_{1\leq j\leq\<\f{r}{10}-\f12\>_{p^2}\atop p\nmid j}\f{1}{j^2}\\
\equiv\ &\sum\limits_{1\leq j\leq2\alpha\atop p\nmid j}\frac{1}{j^2}-\sum\limits_{1\leq j\leq(p^2-1-\alpha)/2\atop p\nmid j}\f{1}{j^2}\\
=\ &\sum_{j=0}^{(-r-1)/2}\f{1}{((p^2+1-\alpha)/2+j)^2}\\
\equiv\ &\sum_{j=0}^{(-r-1)/2}\f{1}{(\f12+\f{r}{10}+j)^2}\pmod{p^2}.
\end{align*}
Now the proof of Lemma \ref{lemma43} is complete.
\end{proof}
Next we present another necessary lemma. To prove this lemma, we need to use Whipple's ${}_4F_3$ transformation formula repeatedly.
\begin{lemma}\label{lem44}
Let $r\leq 1$ be an odd integer coprime with $5$. Then
\begin{align*}
&\sum_{k=0}^{(1-r)/2}\f{(\f{r-1}2)_k(\f{r}5)_k^3}{(1)_k(\f{1}{2}+\f{3r}{10})_k(\f{2r}{5})_k^2}\l(\sum_{j=0}^{k-1}\f{1}{(\f{r}{5}+j)^2}+\sum_{j=0}^{k-1}\f{1}{(\f{2r}{5}+j)^2}\r)\\
=\ &\sum_{k=0}^{(1-r)/2}\f{(\f{r-1}2)_k(\f{r}5)_k^3}{(1)_k(\f{1}{2}+\f{3r}{10})_k(\f{2r}{5})_k^2}\sum_{j=0}^{(-r-1)/2}\f{1}{(\f12+\f{r}{10}+j)^2}.
\end{align*}
\end{lemma}
\begin{proof}Noting that
$$\sum_{j=0}^{(-r-1)/2}\f{1}{(\f12+\f{r}{10}+j)^2}=\sum_{j=0}^{(-r-1)/2}\f{1}{(\f{2r}{5}+j)^2}.$$
Therefore it suffices to prove that
$$\sum_{k=0}^{(1-r)/2}\f{(\f{r-1}2)_k(\f{r}5)_k^3}{(1)_k(\f{1}{2}+\f{3r}{10})_k(\f{2r}{5})_k^2}\l(\sum_{j=0}^{k-1}\f{1}{(\f{r}{5}+j)^2}-\sum_{j=k}^{(-r-1)/2}\f{1}{(\f{2r}{5}+j)^2}\r)=0.$$
Putting $n=\f{1-r}{2}$, $a=\f{r}{5}$, $b=\f{r}{5}-ix$, $c=\f{r}{5}+ix$, $d=\f12+\f{3r}{10}$, $e=\f{2r}{5}-x$, $f=\f{2r}5+x$ in \eqref{4f3}, where $x\in\mathbb{C}$ is a variable, we find that
\begin{align}\label{a1}
&{}_{4}F_3\bigg[\begin{matrix}\f{r-1}{2},&\f{r}5,&\f{r}{5}-ix,&\f{r}{5}+ix\\ &\f12+\f{3r}{10},&\f{2r}5-x,&\f{2r}5+x\end{matrix}\bigg|\ 1\bigg]\notag\\
=\ &\frac{(\f{r}5-x)_n(\f{r}5+x)_n}{(\f{2r}5-x)_n(\f{2r}5+x)_n}\cdot{}_{4}F_3\bigg[\begin{matrix}\f{r-1}{2},&\f{r}{5},&\f12+\f{r}{10}+ix,&\f12+\f{r}{10}-ix\\ &\f12+\f{3r}{10},&\f12+\f{3r}{10}+x,&\f12+\f{3r}{10}-x\end{matrix}\bigg|\ 1\bigg].
\end{align}
Setting $x=0$ in the above equation, we get
\begin{align}\label{a2}
&{}_{4}F_3\bigg[\begin{matrix}\f{r-1}{2},&\f{r}5,&\f{r}{5},&\f{r}{5}\\ &\f12+\f{3r}{10},&\f{2r}5,&\f{2r}5\end{matrix}\bigg|\ 1\bigg]\notag\\
=\ &\frac{(\f{r}5)_n(\f{r}5)_n}{(\f{2r}5)_n(\f{2r}5)_n}\cdot{}_{4}F_3\bigg[\begin{matrix}\f{r-1}{2},&\f{r}{5},&\f12+\f{r}{10},&\f12+\f{r}{10}\\ &\f12+\f{3r}{10},&\f12+\f{3r}{10},&\f12+\f{3r}{10}\end{matrix}\bigg|\ 1\bigg].
\end{align}
Comparing the coefficients of $x^2$ on both sides of \eqref{a1}, we obtain that
\begin{align}\label{a3}
&\sum_{k=0}^{(1-r)/2}\f{(\f{r-1}2)_k(\f{r}5)_k^3}{(1)_k(\f{1}{2}+\f{3r}{10})_k(\f{2r}{5})_k^2}\l(\sum_{j=0}^{k-1}\f{1}{(\f{r}{5}+j)^2}+\sum_{j=0}^{k-1}\f{1}{(\f{2r}{5}+j)^2}\r)\notag\\
=\ &\f{(\f{r}{5})_{\f{1-r}{2}}^2}{(\f{2r}{5})_{\f{1-r}{2}}^2}\sum_{k=0}^{(1-r)/2}\f{(\f{r-1}2)_k(\f{r}5)_k(\f12+\f{r}{10})_k^2}{(1)_k(\f{1}{2}+\f{3r}{10})_k^3}\bigg(\sum_{j=0}^{k-1}\f{1}{(\f12+\f{r}{10}+j)^2}+\sum_{j=0}^{k-1}\f{1}{(\f12+\f{3r}{10}+j)^2}\notag\\
&-\sum_{j=0}^{(-r-1)/2}\f{1}{(\f{r}{5}+j)^2}+\sum_{j=0}^{(-r-1)/2}\f{1}{(\f{2r}{5}+j)^2}\bigg).
\end{align}
Noting that
\begin{equation*}
\sum_{j=0}^{(-r-1)/2}\f{1}{(\f12+\f{3r}{10}+j)^2}=\sum_{j=0}^{(-r-1)/2}\f{1}{(\f{r}{5}+j)^2},
\end{equation*}
and applying \eqref{a2}, \eqref{a3} becomes
\begin{align}\label{a4}
&\sum_{k=0}^{(1-r)/2}\f{(\f{r-1}2)_k(\f{r}5)_k^3}{(1)_k(\f{1}{2}+\f{3r}{10})_k(\f{2r}{5})_k^2}\l(\sum_{j=0}^{k-1}\f{1}{(\f{r}{5}+j)^2}-\sum_{j=k}^{(-r-1)/2}\f{1}{(\f{2r}{5}+j)^2}\r)\notag\\
=\ &\f{(\f{r}{5})_{\f{1-r}{2}}^2}{(\f{2r}{5})_{\f{1-r}{2}}^2}\sum_{k=0}^{(1-r)/2}\f{(\f{r-1}2)_k(\f{r}5)_k(\f12+\f{r}{10})_k^2}{(1)_k(\f{1}{2}+\f{3r}{10})_k^3}\l(\sum_{j=0}^{k-1}\f{1}{(\f12+\f{r}{10}+j)^2}-\sum_{j=k}^{(-r-1)/2}\f{1}{(\f12+\f{3r}{10}+j)^2}\r).
\end{align}
On the other hand,
\begin{align}\label{a5}
&\f{(\f{r}{5})_{\f{1-r}{2}}^2}{(\f{2r}{5})_{\f{1-r}{2}}^2}\sum_{k=0}^{(1-r)/2}\f{(\f{r-1}2)_k(\f{r}5)_k(\f12+\f{r}{10})_k^2}{(1)_k(\f{1}{2}+\f{3r}{10})_k^3}\l(\sum_{j=0}^{k-1}\f{1}{(\f12+\f{r}{10}+j)^2}-\sum_{j=k}^{(-r-1)/2}\f{1}{(\f12+\f{3r}{10}+j)^2}\r)\notag\\
=\ &\f{(\f{r}{5})_{\f{1-r}{2}}^2}{(\f{2r}{5})_{\f{1-r}{2}}^2}\sum_{k=0}^{(1-r)/2}\f{(\f{r-1}2)_{\f{1-r}{2}-k}(\f{r}5)_{\f{1-r}{2}-k}(\f12+\f{r}{10})_{\f{1-r}{2}-k}^2}{(1)_{\f{1-r}{2}-k}(\f{1}{2}+\f{3r}{10})_{\f{1-r}{2}-k}^3}\bigg(\sum_{j=0}^{(-r-1)/2-k}\f{1}{(\f12+\f{r}{10}+j)^2}\notag\\
&-\sum_{j=(1-r)/2-k}^{(-r-1)/2}\f{1}{(\f12+\f{3r}{10}+j)^2}\bigg)\notag\\
=\ &\f{(\f{r-1}{2})_{\f{1-r}{2}}(\f12+\f{r}{10})_{\f{1-r}{2}}^2(\f{r}{5})_{\f{1-r}{2}}^3}{(1)_{\f{1-r}{2}}(\f{2r}{5})_{\f{1-r}{2}}^2(\f12+\f{3r}{10})_{\f{1-r}{2}}^3}\sum_{k=0}^{(1-r)/2}\f{(\f{r-1}2)_k(\f{r}5)_k^3}{(1)_k(\f{1}{2}+\f{3r}{10})_k(\f{2r}{5})_k^2}\l(\sum_{j=k}^{(-r-1)/2}\f{1}{(\f{2r}{5}+j)^2}-\sum_{j=0}^{k-1}\f{1}{(\f{r}{5}+j)^2}\r)\notag\\
=\ &\sum_{k=0}^{(1-r)/2}\f{(\f{r-1}2)_k(\f{r}5)_k^3}{(1)_k(\f{1}{2}+\f{3r}{10})_k(\f{2r}{5})_k^2}\l(\sum_{j=k}^{(-r-1)/2}\f{1}{(\f{2r}{5}+j)^2}-\sum_{j=0}^{k-1}\f{1}{(\f{r}{5}+j)^2}\r),
\end{align}
where in the last step we have used the facts that $(\f{r-1}{2})_{\f{1-r}{2}}=(-1)^{(1-r)/2}(1)_{\f{1-r}{2}}$, $(\f12+\f{r}{10})_{\f{1-r}{2}}=(-1)^{(1-r)/2}(\f{2r}{5})_{\f{1-r}{2}}$ and $(\f{r}{5})_{\f{1-r}{2}}=(-1)^{(1-r)/2}(\f12+\f{3r}{10})_{\f{1-r}{2}}$. Combining \eqref{a4} and \eqref{a5}, we obtain
$$\sum_{k=0}^{(1-r)/2}\f{(\f{r-1}2)_k(\f{r}5)_k^3}{(1)_k(\f{1}{2}+\f{3r}{10})_k(\f{2r}{5})_k^2}\l(\sum_{j=0}^{k-1}\f{1}{(\f{r}{5}+j)^2}-\sum_{j=k}^{(-r-1)/2}\f{1}{(\f{2r}{5}+j)^2}\r)=0.$$
Now the proof of the desired result is complete.
\end{proof}

\medskip

\noindent{\it Proof of Theorem \ref{theorem1}}. Putting $m=\frac{1-r}{2}$, $t=\f{r}{5}$, $n=\frac{p-r}{5}$, $a=\f{r-5}{10}$, $b=-\f{p}{5}$ and $c=-\f{ip}{5}$ in \eqref{new}, we have
\begin{align}\label{new3}
&{}_7F_6\bigg[\begin{matrix}\frac{r}{5},&1+\frac{r}{10},&\frac{r-p}{5},&\frac{r+5}{10},&\frac{r+p}{5},&\frac{r+ip}{5},&\frac{r-ip}{3}\\ &\f r{10},&1+\f p5,&\f{r+5}{10},&1-\f p5,&1-\f{ip}{5},&1+\f{ip}{5}\end{matrix}\bigg|\ 1\bigg]\notag\\
=\ &{}_4F_3\bigg[\begin{matrix}\f{r-1}{2},&\f{r-p}{5},&\f{r+p}{5},&\f{r}{5}\\ &\f{2r-ip}{5},&\f{1}{2}+\f{3r}{10},&\f{2r+ip}{5}\end{matrix}\bigg|\ 1\bigg]\notag\\
&\times\f{(1+\f{r}{5})_n(\f12-\f{r}{10}-\f{p}5)_n(\f{2r}{5}-\f{ip}5)_n(1-\f{r}5-\f{(1+i)p}{5})_n}{(\f12+\f{r}{10})_n(1-\frac{p}{5})_n(1-\f{ip}{5})_n(\f{r}{5}-\f{(1+i)p}{5})_n}.
\end{align}
With the help of Lemmas \ref{lemw} and \ref{lemma32}, the left-hand side of \eqref{new3} becomes
\begin{align}\label{t1}
&{}_7F_6\bigg[\begin{matrix}\frac{r}{5},&1+\frac{r}{10},&\frac{r-p}{5},&\frac{r+5}{10},&\frac{r+p}{5},&\frac{r+ip}{5},&\frac{r-ip}{3}\\ &\f r{10},&1+\f p5,&\f{r+5}{10},&1-\f p5,&1-\f{ip}{5},&1+\f{ip}{5}\end{matrix}\bigg|\ 1\bigg]_{\f{p-r}{5}}\notag\\
\equiv\ &\f1r\sum_{k=0}^{(p-r)/5}(10k+r)\f{(\f r5)_k^5}{(1)_k^5}\l(1+\f{1}{625}p^4\sum_{j=0}^{k-1}\f{1}{(r/5+j)^4}-\f{1}{625}p^4\sum_{j=0}^{k-1}\f1{(1+j)^4}\r)\notag\\
\eq\ &\f1r\sum_{k=0}^{(p-r)/5}(10k+r)\f{(\f r5)_k^5}{(1)_k^5}\notag\\
\eq\ &\f1r\sum_{k=0}^{p-1}(10k+r)\f{(\f r5)_k^5}{(1)_k^5}\pmod{p^5},
\end{align}
where in the last step we have used the fact that $(\f{r}5)_k\equiv0\pmod p$ for $k\in\{(p-r)/5+1,\ldots,p-1\}$.

Now we consider the right-hand side of \eqref{new3} modulo $p^5$. Firstly, by Lemma \ref{lem44}, we have
\begin{align}\label{s1}
&{}_4F_3\bigg[\begin{matrix}\f{r-1}{2},&\f{r-p}{5},&\f{r+p}{5},&\f{r}{5}\\ &\f{2r-ip}{5},&\f{1}{2}+\f{3r}{10},&\f{2r+ip}{5}\end{matrix}\bigg|\ 1\bigg]\notag\\
=\ &{}_4F_3\bigg[\begin{matrix}\f{r-1}{2},&\f{r-p}{5},&\f{r+p}{5},&\f{r}{5}\\ &\f{2r-ip}{5},&\f{1}{2}+\f{3r}{10},&\f{2r+ip}{5}\end{matrix}\bigg|\ 1\bigg]_{\f{1-r}{2}}\notag\\
\equiv\ &\sum_{k=0}^{(1-r)/2}\f{(\f{r-1}2)_k(\f{r}5)_k^3}{(1)_k(\f{1}{2}+\f{3r}{10})_k(\f{2r}{5})_k^2}\l(1-\f{p^2}{25}\sum_{j=0}^{k-1}\f{1}{(\f{r}{5}+j)^2}-\f{p^2}{25}\sum_{j=0}^{k-1}\f{1}{(\f{2r}{5}+j)^2}\r)\notag\\
=\ &\sum_{k=0}^{(1-r)/2}\f{(\f{r-1}2)_k(\f{r}5)_k^3}{(1)_k(\f{1}{2}+\f{3r}{10})_k(\f{2r}{5})_k^2}\l(1-\f{p^2}{25}\sum_{j=0}^{(-r-1)/2}\f{1}{(\f12+\f{r}{10}+j)^2}\r)\pmod{p^4}.
\end{align}
Note that
\begin{align*}
\l(1+\f{r}{5}\r)_n=\ &(-1)^n\f{p}{r}\l(1-\f{p}{5}\r)_n,\\
\l(\f12-\f{r}{10}-\f{p}{5}\r)_n=\ &(-1)^n\l(\f12+\f{3r}{10}\r)_n,\\
\l(1-\f{r}5-\f{(1+i)p}{5}\r)_n=\ &(-1)^n\l(\f{2r}{5}+\f{ip}5\r)_n,\\
\l(\f{r}{5}-\f{(1+i)p}{5}\r)_n=\ &(-1)^n\l(1+\f{ip}{5}\r)_n.
\end{align*}
Then we have
\begin{align}\label{s2}
&\f{(1+\f{r}{5})_n(\f12-\f{r}{10}-\f{p}5)_n(\f{2r}{5}-\f{ip}5)_n(1-\f{r}5-\f{(1+i)p}{5})_n}{(\f12+\f{r}{10})_n(1-\frac{p}{5})_n(1-\f{ip}{5})_n(\f{r}{5}-\f{(1+i)p}{5})_n}=\f{p}{r}\cdot\f{(\f12+\f{3r}{10})_n(\f{2r}{5}-\f{ip}5)_n(\f{2r}{5}+\f{ip}5)_n}{(\f12+\f{r}{10})_n(1-\f{ip}{5})_n(1+\f{ip}{5})_n}\notag\\
\equiv\ &\f{p}{r}\cdot\f{(\f12+\f{3r}{10})_n(\f{2r}{5})_n^2}{(\f12+\f{r}{10})_n(1)_n^2}\l(1+\f{p^2}{25}\sum_{j=0}^{(p-r)/5-1}\f{1}{(\f{2r}5+j)^2}-\f{p^2}{25}\sum_{j=0}^{(p-r)/5-1}\f{1}{(1+j)^2}\r)\pmod{p^5}.
\end{align}
Furthermore, in light of Lemmas \ref{robert} and \ref{g},
\begin{align}\label{s3}
&\f{(\f12+\f{3r}{10})_n(\f{2r}{5})_n^2}{(\f12+\f{r}{10})_n(1)_n^2}=\f{\Gamma_p(\f12+\f{r}{10})\Gamma_p(\f12+\f{r}{10}+\f{p}5)\Gamma_p(\f{r}{5}+\f{p}{5})^2}{\Gamma_p(\f12+\f{3r}{10})\Gamma_p(\f12-\f{r}{10}+\f{p}5)\Gamma_p(1-\f{r}5+\f{p}5)^2\Gamma_p(\f{2r}5)^2}\notag\\
=\ &\f{\Gamma_p(\f{r}{5}+\f{p}{5})^2\Gamma_p(\f{r}5-\f{p}5)^2}{\Gamma_p(\f{2r}5)^2\Gamma_p(\f12+\f{3r}{10})\Gamma_p(\f12-\f{r}{10})\Gamma_p(\f12-\f{r}{10}-\f{p}5)\Gamma_p(\f12-\f{r}{10}+\f{p}5)}\notag\\
\equiv\ &\f{\Gamma_p(\f{r}{5})^4}{\Gamma_p(\f{2r}5)^2\Gamma_p(\f12+\f{3r}{10})\Gamma_p(\f12-\f{r}{10})^3}\l(1+\f{2p^2}{25}\sum\limits_{1\leq j\leq\<-\f{r}{5}\>_{p^2}\atop p\nmid j}\f{1}{j^2}-\f{p^2}{25}\sum\limits_{1\leq j\leq\<\f{r}{10}-\f12\>_{p^2}\atop p\nmid j}\f{1}{j^2}\r)\pmod{p^4}.
\end{align}
Substituting \eqref{s1}--\eqref{s3} into the right-hand side of \eqref{new3} and applying Lemma \ref{lemma33}, we have
\begin{align}\label{t2}
&{}_4F_3\bigg[\begin{matrix}\f{r-1}{2},&\f{r-p}{5},&\f{r+p}{5},&\f{r}{5}\\ &\f{2r-ip}{5},&\f{1}{2}+\f{3r}{10},&\f{2r+ip}{5}\end{matrix}\bigg|\ 1\bigg]\times\f{(1+\f{r}{5})_n(\f12-\f{r}{10}-\f{p}5)_n(\f{2r}{5}-\f{ip}5)_n(1-\f{r}5-\f{(1+i)p}{5})_n}{(\f12+\f{r}{10})_n(1-\frac{p}{5})_n(1-\f{ip}{5})_n(\f{r}{5}-\f{(1+i)p}{5})_n}\notag\\
\equiv\ &\f{p\Gamma_p(\f{r}{5})^4}{r\Gamma_p(\f{2r}5)^2\Gamma_p(\f12+\f{3r}{10})\Gamma_p(\f12-\f{r}{10})^3}\sum_{k=0}^{(1-r)/2}\f{(\f{r-1}2)_k(\f{r}5)_k^3}{(1)_k(\f{1}{2}+\f{3r}{10})_k(\f{2r}{5})_k^2}\bigg(1-\f{p^2}{25}\sum_{j=0}^{(-r-1)/2}\f{1}{(\f12+\f{r}{10}+j)^2}\notag\\
&+\f{p^2}{25}\sum_{j=0}^{(p-r)/5-1}\f{1}{(\f{2r}5+j)^2}-\f{p^2}{25}\sum_{j=0}^{(p-r)/5-1}\f{1}{(1+j)^2}+\f{2p^2}{25}\sum\limits_{1\leq j\leq\<-\f{r}{5}\>_{p^2}\atop p\nmid j}\f{1}{j^2}-\f{p^2}{25}\sum\limits_{1\leq j\leq\<\f{r}{10}-\f12\>_{p^2}\atop p\nmid j}\f{1}{j^2}\bigg)\notag\\
\equiv\ &\f{p\Gamma_p(\f{r}{5})^4}{r\Gamma_p(\f{2r}5)^2\Gamma_p(\f12+\f{3r}{10})\Gamma_p(\f12-\f{r}{10})^3}\sum_{k=0}^{(1-r)/2}\f{(\f{r-1}2)_k(\f{r}5)_k^3}{(1)_k(\f{1}{2}+\f{3r}{10})_k(\f{2r}{5})_k^2}\pmod{p^5}.
\end{align}
Then the desired result follows from \eqref{new3}, \eqref{t1} and \eqref{t2}. So far we have completed the proof of Theorem \ref{theorem1}. \qed

\section{Proof of Theorem \ref{theorem2}}

Similarly to Lemmas \ref{lemma31}, \ref{lemma32} and \ref{lemma33}, we can deduce the following three lemmas, so we omit the proofs.
\begin{lemma}\label{lemma41}
Under the assumptions of congruence \eqref{GLSconj2eq}, we have
$$\sum_{k=0}^{(p-r)/3}(6k+r)\frac{(\frac{r}{3})_k^6}{(1)_k^6}\equiv0\pmod{p}.$$
\end{lemma}

\begin{lemma}\label{lemma42}
Under the assumptions of congruence \eqref{GLSconj2eq}, we have
$$\sum_{k=0}^{(p-r)/3}(6k+r)\frac{(\frac{r}{3})_k^6}{(1)_k^6}\left(\sum_{j=0}^{k-1}\frac{1}{(r/3+j)^4}-\sum_{j=0}^{k-1}\frac{1}{(1+j)^4}\right)\equiv0\pmod p.$$
\end{lemma}

\begin{lemma}\label{lemma43}
Under the assumptions of congruence \eqref{GLSconj2eq}, we have
$$\sum_{j=0}^{(p-r)/3-1}\frac{1}{(\f{2r}{3}+j)^2}-\sum_{j=0}^{(p-r)/3-1}\frac{1}{(1+j)^2}+3\sum\limits_{1\leq j\leq\<-r/3\>_{p^2}\atop p\nmid j}\frac{1}{j^2}\equiv\sum_{j=0}^{-r}\frac{1}{(\f{r}3+j)^2}\pmod{p^2}.$$
\end{lemma}

In order to prove Theorem \ref{theorem1}, we still need the next lemma.
\begin{lemma}\label{lemma44}
Let $r\leq 1$ be an integer coprime with 3. Then
\begin{equation}\label{equ3r}
\sum_{k=0}^{1-r}\f{(r-1)_k(\f r3)_k^3}{(1)_k(\f{2r}3)_k^3}\l(\sum_{j=0}^{k-1}\f{1}{(\f{r}3+j)^2}+\sum_{j=0}^{k-1}\f{1}{(\f{2r}3+j)^2}\r)=\sum_{k=0}^{1-r}\f{(r-1)_k(\f r3)_k^3}{(1)_k(\f{2r}3)_k^3}\sum_{j=0}^{-r}\f{1}{(\f{r}3+j)^2}.
\end{equation}
\end{lemma}
\begin{proof}
Clearly,
$$
(r-1)_{1-r}=(-1)^{1-r}(1)_{1-r}\quad \t{and}\quad \l(\f r3\r)_{1-r}=(-1)^{1-r}\l(\f{2r}3\r)_{1-r}.
$$
Then we have
\begin{align*}
\sum_{k=0}^{1-r}\f{(r-1)_k(\f r3)_k^3}{(1)_k(\f{2r}3)_k^3}\sum_{j=0}^{k-1}\f{1}{(\f{2r}3+j)^2}=\ &\sum_{k=0}^{1-r}\f{(r-1)_{1-r-k}(\f r3)_{1-r-k}^3}{(1)_{1-r-k}(\f{2r}3)_{1-r-k}^3}\sum_{j=0}^{-r-k}\f{1}{(\f{2r}3+j)^2}\\
=\ &\f{(r-1)_{1-r}(\f r3)_{1-r}^3}{(1)_{1-r}(\f{2r}3)_{1-r}^3}\sum_{k=0}^{1-r}\f{(r-1)_k(\f r3)_k^3}{(1)_k(\f{2r}3)_k^3}\sum_{j=k}^{-r}\f{1}{(\f{2r}3-r-j)^2}\\
=\ &\sum_{k=0}^{1-r}\f{(r-1)_k(\f r3)_k^3}{(1)_k(\f{2r}3)_k^3}\sum_{j=k}^{-r}\f{1}{(\f{r}3+j)^2}.
\end{align*}
Therefore
\begin{align*}
&\sum_{k=0}^{1-r}\f{(r-1)_k(\f r3)_k^3}{(1)_k(\f{2r}3)_k^3}\l(\sum_{j=0}^{k-1}\f{1}{(\f{r}3+j)^2}+\sum_{j=0}^{k-1}\f{1}{(\f{2r}3+j)^2}\r)\\
=\ &\sum_{k=0}^{1-r}\f{(r-1)_k(\f r3)_k^3}{(1)_k(\f{2r}3)_k^3}\l(\sum_{j=0}^{k-1}\f{1}{(\f{r}3+j)^2}+\sum_{j=k}^{-r}\f{1}{(\f{r}3+j)^2}\r)\\
=\ &\sum_{k=0}^{1-r}\f{(r-1)_k(\f r3)_k^3}{(1)_k(\f{2r}3)_k^3}\sum_{j=0}^{-r}\f{1}{(\f{r}3+j)^2}.
\end{align*}
\end{proof}

\medskip

\noindent{\it Proof of Theorem \ref{theorem2}}. Setting $m=1-r$, $t=\frac{r}{3}$, $n=\frac{p-r}{3}$, $a=0$, $b=-\frac{p}{3}$, $c=\frac{ip}{3}$ in \eqref{new}, we have
\begin{align}\label{new2}
&{}_7F_6\bigg[\begin{matrix}\frac{r}{3},&1+\frac{r}{6},&\frac{r-p}{3},&\frac{r}{3},&\frac{r+p}{3},&\frac{r-ip}{3},&\frac{r+ip}{3}\\ &\f r6,&1+\f p3,&1,&1-\f p3,&1+\f{ip}{3},&1-\f{ip}{3}\end{matrix}\bigg|\ 1\bigg]\notag\\
=\ &{}_4F_3\bigg[\begin{matrix}r-1,&\f{r-p}{3},&\f{r+p}{3},&\f{r}{3}\\ &\f{2r+ip}{3},&\f{2r}{3},&\f{2r-ip}{3}\end{matrix}\bigg|\ 1\bigg]\notag\\
&\times\f{(1+\f{r}{3})_n(\f{2r+ip}{3})_n(1-\f{r+p}{3})_n(1-\f{r+p-ip}{3})_n}{(1)_n(1-\frac{p}{3})_n(1+\f{ip}{3})_n(\f{r-p+ip}{3})_n}.
\end{align}
In view of Lemmas \ref{lemw} and \ref{lemma42},
\begin{align}\label{h1}
&{}_7F_6\bigg[\begin{matrix}\frac{r}{3},&1+\frac{r}{6},&\frac{r-p}{3},&\frac{r}{3},&\frac{r+p}{3},&\frac{r-ip}{3},&\frac{r+ip}{3}\\ &\f r6,&1+\f p3,&1,&1-\f p3,&1+\f{ip}{3},&1-\f{ip}{3}\end{matrix}\bigg|\ 1\bigg]\notag\\
=\ &{}_7F_6\bigg[\begin{matrix}\frac{r}{3},&1+\frac{r}{6},&\frac{r-p}{3},&\frac{r}{3},&\frac{r+p}{3},&\frac{r-ip}{3},&\frac{r+ip}{3}\\ &\f r6,&1+\f p3,&1,&1-\f p3,&1+\f{ip}{3},&1-\f{ip}{3}\end{matrix}\bigg|\ 1\bigg]_{\f{p-r}{3}}\notag\\
\eq\ &\f1r\sum_{k=0}^{(p-r)/3}(6k+r)\f{(\f r3)_k^6}{(1)_k^6}\l(1+\f{1}{81}p^4\sum_{j=0}^{k-1}\f{1}{(r/3+j)^4}-\f{1}{81}p^4\sum_{j=0}^{k-1}\f1{(1+j)^4}\r)\notag\\
\eq\ &\f1r\sum_{k=0}^{(p-r)/3}(6k+r)\f{(\f r3)_k^6}{(1)_k^6}\notag\\
\eq\ &\f1r\sum_{k=0}^{p-1}(6k+r)\f{(\f r3)_k^6}{(1)_k^6}\pmod{p^5},
\end{align}
where the last step follows from the fact $(\f{r}{3})_k\equiv0\pmod p$ for $k\in\{(p-r)/3+1,\ldots,p-1\}$.

Now we consider the right-hand side of \eqref{new2} modulo $p^5$. Clearly, for any $\alpha$, $t\in\mathbb{Z}_p$,
\begin{align*}
(\alpha+tp)_k(\alpha-tp)_k=\ &(\alpha^2-t^2p^2)((\alpha+1)^2-t^2p^2)\cdots((\alpha+k-1)^2-t^2p^2)\\
\equiv\ &(\alpha)_k^2\left(1-t^2p^2\sum_{j=0}^{k-1}\f{1}{(\alpha+j)^2}\right)\pmod{p^4}.
\end{align*}
Then, by \eqref{equ3r}, we get
\begin{align}\label{e1}
&{}_4F_3\bigg[\begin{matrix}r-1,&\f{r-p}{3},&\f{r+p}{3},&\f{r}{3}\\ &\f{2r+ip}{3},&\f{2r}{3},&\f{2r-ip}{3}\end{matrix}\bigg|\ 1\bigg]\notag\\
=\ &{}_4F_3\bigg[\begin{matrix}r-1,&\f{r-p}{3},&\f{r+p}{3},&\f{r}{3}\\ &\f{2r+ip}{3},&\f{2r}{3},&\f{2r-ip}{3}\end{matrix}\bigg|\ 1\bigg]_{1-r}\notag\\
\equiv\ &\sum_{k=0}^{1-r}\f{(r-1)_k(\f{r}3)_k^3}{(1)_k(\f{2r}{3})_k^3}\l(1-\f{p^2}{9}\sum_{j=0}^{k-1}\f{1}{(\f{r}{3}+j)^2}-\f{p^2}{9}\sum_{j=0}^{k-1}\f{1}{(\f{2r}{3}+j)^2}\r)\notag\\
=\ &\sum_{k=0}^{1-r}\f{(r-1)_k(\f{r}3)_k^3}{(1)_k(\f{2r}{3})_k^3}\l(1-\f{p^2}{9}\sum_{j=0}^{-r}\f{1}{(\f{r}{3}+j)^2}\r)
\pmod{p^4}.
\end{align}
Note that
\begin{align*}
\l(1+\f{r}{3}\r)_n=\ &(-1)^n\f{p}{r}\l(1-\f{p}{3}\r)_n,\\
\l(1-\f{r+p}{3}\r)_n=\ &(-1)^n\l(\f{2r}{3}\r)_n,\\
\l(1-\f{r+p-ip}{3}\r)_n=\ &(-1)^n\l(\f{2r-ip}{3}\r)_n,\\
\l(\f{r-p+ip}{3}\r)_n=\ &(-1)^n\l(1-\f{ip}{3}\r)_n.\\
\end{align*}
Hence, we have
\begin{align}\label{e2}
&\f{(1+\f{r}{3})_n(\f{2r+ip}{3})_n(1-\f{r+p}{3})_n(1-\f{r+p-ip}{3})_n}{(1)_n(1-\frac{p}{3})_n(1+\f{ip}{3})_n(\f{r-p+ip}{3})_n}=\frac{p}{r}\cdot\frac{(\f{2r}{3})_n(\f{2r+ip}{3})_n(\f{2r-ip}{3})_n}{(1)_n(1+\f{ip}{3})_n(1-\f{ip}{3})_n}\notag\\
\equiv\ &\frac{p}{r}\cdot\frac{(\f{2r}{3})_n^3}{(1)_n^3}\l(1+\f{p^2}{9}\sum_{j=0}^{(p-r)/3-1}\f{1}{(\f{2r}{3}+j)^2}-\f{p^2}{9}\sum_{j=0}^{(p-r)/3-1}\f{1}{(1+j)^2}\r)\pmod{p^5}.
\end{align}
According to Lemmas \ref{robert} and \ref{g}, we obtain
\begin{align}\label{e3}
\frac{(\f{2r}{3})_n^3}{(1)_n^3}=\ & \frac{\Gamma_p(\f{r+p}{3})^3\Gamma_p(1)^3}{\Gamma_p(\f{2r}{3})^3\Gamma_p(1+\f{p-r}{3})^3}=(-1)^{r+1}\frac{\Gamma_p(\f{r+p}{3})^3\Gamma_p(\f{r-p}{3})^3}{\Gamma_p(\f{2r}{3})^3}\notag\\
\equiv\ &(-1)^{r+1}\f{\Gamma_p(\f r3)^6}{\Gamma_p(\f{2r}3)^3}\l(1+\f{p^2}{3}\sum\limits_{1\leq j\leq\<-\f{r}{3}\>_{p^2}\atop p\nmid j}\f{1}{j^2}\r)\pmod{p^4}.
\end{align}
Substituting \eqref{e1}--\eqref{e3} into the right-hand side of \eqref{new2} and applying Lemma \ref{lemma43}, we arrive at
\begin{align}\label{h2}
&{}_4F_3\bigg[\begin{matrix}r-1,&\f{r-p}{3},&\f{r+p}{3},&\f{r}{3}\\ &\f{2r+ip}{3},&\f{2r}{3},&\f{2r-ip}{3}\end{matrix}\bigg|\ 1\bigg]\times\f{(1+\f{r}{3})_n(\f{2r+ip}{3})_n(1-\f{r+p}{3})_n(1-\f{r+p-ip}{3})_n}{(1)_n(1-\frac{p}{3})_n(1+\f{ip}{3})_n(\f{r-p+ip}{3})_n}\notag\\
\equiv\ &(-1)^{r+1}\f{p}{r}\cdot\f{\Gamma_p(\frac{r}{3})^6}{\Gamma_p(\f{2r}3)^3}\sum_{k=0}^{1-r}\f{(r-1)_k(\f{r}3)_k^3}{(1)_k(\f{2r}{3})_k^3}\bigg{(}1-\f{p^2}{9}\sum_{j=0}^{-r}\f{1}{(\f{r}{3}+j)^2}\notag\\
&+\f{p^2}{9}\sum_{j=0}^{(p-r)/3-1}\f{1}{(\f{2r}{3}+j)^2}-\f{p^2}{9}\sum_{j=0}^{(p-r)/3-1}\f{1}{(1+j)^2}+\f{p^2}{3}\sum\limits_{1\leq j\leq\<-\f{r}{3}\>_{p^2}\atop p\nmid j}\f{1}{j^2}\bigg{)}\notag\\
\equiv\ &(-1)^{r+1}\f{p}{r}\cdot\f{\Gamma_p(\frac{r}{3})^6}{\Gamma_p(\f{2r}3)^3}\sum_{k=0}^{1-r}\f{(r-1)_k(\f{r}3)_k^3}{(1)_k(\f{2r}{3})_k^3}\pmod{p^5}.
\end{align}
Combining \eqref{new2}, \eqref{h1} and \eqref{h2}, we finally complete the proof of Theorem \ref{theorem2}.\qed

\begin{Acks}
The first author is supported by the National Natural Science Foundation of China (grant 11971222). The second author is supported by the National Natural Science Foundation of China (grant 12201301).
\end{Acks}


\begin{thebibliography}{99}
\bibitem{Andrews} G.E. Andrews, R. Askey, R.Roy, {\it Special Functions}, Encyclopedia Math. Appl., Gambridge and Univ. Press, Cambridge, 1999.
\bibitem{GuoLiu} V.J.W. Guo and J.-C. Liu, {\it Some congruences related to a congruence of Van Hamme}, Integral Transforms Spec. Funct. {\bf 31} (2020), 221--231.
\bibitem{GLS} V.J.W. Guo, J.-C. Liu and M.J. Schlosser, {\it An extension of a supercongruence of Long and Ramakrishna}, Proc. Amer. Math. Soc. {\bf 151} (2023), 1157--1166.
\bibitem{GS1} V.J.W. Guo and M.J. Schlosser, {\it A family of $q$-hypergeometric congruences modulo the fourth power of a cyclotomic polynomial}, Israel J. Math. {\bf 240} (2020), 821--835.
\bibitem{GS} V.J.W. Guo and M.J. Schlosser, {\it Some $q$-supercongruences from transformation formulas for basic hypergeometric series}, Constr. Approx. {\bf 53} (2021), 155--200.
\bibitem{Liu2017} J.-C. Liu, {\it A $p$-adic supercongruence for truncated hypergeoemtric series ${}_7F_6$}, Results Math. {\bf72} (2017), 2057--2066.
\bibitem{Liu2021} J.-C. Liu, {\it Supercongruences arising from transformations of hypergeometric series}, J. Math. Anal. Appl. {\bf 497} (2021), Art. 124915.
\bibitem{LR} L. Long and R. Ramakrishna, {\it Some supercongruences occurring in truncated hypergeometric series}, Adv. Math. {\bf 290} (2016), 773--808.
\bibitem{MaoPan2022} G.-S. Mao and H. Pan, {\it Congruences corresponding to hypergeometric identities I. ${}_2F_1$ transformations}, J. Math. Anal. Appl. {\bf 505} (2022), Art. 125527.
\bibitem{Morita} Y. Morita, {\it A $p$-adic analogue of the $\Gamma$-function}, J. Fac. Sci. Univ. Tokyo Sect. IA Math. {\bf22} (1975), no. 2, 255--266.
\bibitem{PTW} H. Pan, R. Tauraso and C.Wang, {\it A local-global theorem for $p$-adic supercongruences}, J. Reine Angew. Math. {\bf 790} (2022), 53--83.
\bibitem{Robert00} A.M. Robert, {\it A Course in $p$-Adic Analysis}, Graduate Texts in Mathematics, Vol. {\bf 198}, Springer-Verlag, New York, 2000.
\bibitem{Slavutsky} I.Sh. Slavutsky, {\it Leudesdorf's theorem and Bernoulli numbers}, Arch. Math. {\bf 35} (1999), 299--303.
\bibitem{Sun2011} Z.-W. Sun, {\it Super congruences and Euler numbers}, Sci. China Math. {\bf 54} (2011), 2509--2535.
\bibitem{VanHamme} L. Van Hamme, {\it Some conjectures concerning partial sums of generalized hypergeometric series}, $p$-adic functional analysis (Nijmegen, 1996), Lecture Notes in Pure and Appl. Math. {\bf 192}, Dekker, New York (1997), 223--236.
\bibitem{Wang} C. Wang, {\it Supercongruences arising from a $_7F_6$ hypergeometric transformation formula}, preprint, arXiv:2306.02635.
\bibitem{WP} C. Wang and H. Pan, {\it Supercongruences concerning truncated hypergeometric series}, Math. Z. {\bf 300} (2022), 161--177.
\bibitem{WS} C. Wang and Z.-W. Sun, {\it Proof of some conjectural hypergeometric supercongruences via curious identities}, J. Math. Anal. Appl. {\bf 505} (2022), Art. 125575.
\end{thebibliography}
\end{document}